\documentclass[nobib]{tufte-handout}
\usepackage{amsmath,stmaryrd,amssymb,amsthm,url,booktabs,hyperref}

\newcommand{\doi}[1]{doi:\,\href{http://doi.org/#1}{#1}}
\newcommand{\mrev}[1]{\href{http://www.ams.org/mathscinet-getitem?mr=#1}{MR#1}}
\newcommand{\zbl}[1]{\href{http://www.emis.de/cgi-bin/MATH-item?#1}{Zbl #1}}
\newcommand{\arx}[1]{\href{http://arXiv.org/abs/#1}{arXiv:#1}}

\title{Homotopy types of topological stacks of categories}
\author[D.M. Roberts]{David Michael Roberts\thanks{%
\url{david.roberts@adelaide.edu.au}\\\medskip
\href{https://orcid.org/0000-0002-3478-0522}{orcid.org/0000-0002-3478-0522}\\\medskip
School of Computer and Mathematical Sciences, the University of Adelaide, SA 5005.\\\medskip
This work was supported in 2008--09 by an Australian Postgraduate Award.\\\medskip 
This document is released under a \href{https://creativecommons.org/licenses/by/4.0/}{CC BY 4.0 license}.
}}
\date{11 June 2024}

\usepackage{euler}

\usepackage{ifpdf}
\ifpdf
  \usepackage[all,pdf]{xy} %<-- this should be loaded as last package
\else
  \input xy %<--- so that the arXiv can generate a .ps file
  \xyoption{all}
  \xyoption{2cell} 
  \xyoption{v2}
\fi

\SelectTips{cm}{11} % Makes all xy arrows match the usual LaTeX arrows

\theoremstyle{plain}
\newtheorem{thm}{Theorem}
\newtheorem{lemma}[thm]{Lemma}
\newtheorem{propn}[thm]{Proposition}
\newtheorem{corollary}[thm]{Corollary}

\newtheorem*{thma}{\textbf{Theorem A}}
\newtheorem*{thmaprime}{\textbf{Theorem A'}}

\theoremstyle{definition}
\newtheorem{defn}[thm]{Definition}
\newtheorem{example}[thm]{Example}

\newcommand{\To}{\longrightarrow}
\newcommand{\ST}{\rightrightarrows} 
\newcommand{\into}{\hookrightarrow}

\newcommand{\Top}{\textbf{Top}}
\newcommand{\CGH}{\textbf{CGH}}
\newcommand{\Set}{\textbf{Set}}
\newcommand{\Cat}{\textbf{Cat}}
\newcommand{\St}{\textbf{\textrm{St}}}

\DeclareMathOperator{\Obj}{Obj}
\DeclareMathOperator{\Mor}{Mor}

\DeclareMathOperator{\pr}{pr}
\DeclareMathOperator{\dom}{dom}
\DeclareMathOperator{\cod}{cod}

\DeclareMathOperator{\codisc}{codisc}
\DeclareMathOperator{\disc}{disc}
\DeclareMathOperator{\id}{id}
\DeclareMathOperator{\op}{op}

\def\sl {\downarrow}

\DeclareFontFamily{U}{min}{}
\DeclareFontShape{U}{min}{m}{n}{<-> udmj30}{}
\newcommand\yon{\!\text{\usefont{U}{min}{m}{n}\symbol{'210}}\!}

\begin{document}

\maketitle

\begin{abstract}
This note extends Quillen's Theorem A to a large class of categories internal to topological spaces. 
This allows us to show that under a mild condition a fully faithful and essentially surjective functor between such topological categories induces a homotopy equivalence of classifying spaces.
It follows from this that we can associate a 2-functorial homotopy type to a wide class of topological stacks of categories, taking values in the 2-category of spaces, continuous maps and homotopy classes of homotopies of maps. 
This generalises work of Noohi and Ebert on the homotopy types of topological stacks of groupoids under the restriction to the site with \emph{numerable} open covers.
\end{abstract}

It is well known that a category gives rise to a CW complex---its classifying space---and thus represents a homotopy type. In fact, any 
CW complex can be represented (up to homotopy) as the classifying space of a category. 
It is therefore of interest to know when a functor induces a homotopy equivalence of classifying spaces, and Quillen's Theorem A \cite{Quillen_73} answers this question for us: if all the geometric realisations of the lax fibres of a functor are contractible, the functor induces a homotopy equivalence.

However, there are homotopy types that are best realised as the classifying spaces of \emph{topological} categories, that is, categories internal to $\Top$, or some cartesian closed variant. 
Examples include the Borel construction $X \times_K EK$ for a space with the action of a topological group or monoid $K$. 
It is therefore natural to try to extend Theorem A to topological categories.

The original formulation of Quillen's Theorem A has the hypothesis that a family of spaces, indexed by a set, are each contractible. 
If one writes this down verbatim for topological categories, the hypothesis turns into ``such and such a map has contractible fibres", which isn't even sufficient to tell us that a map of topological spaces is a (weak) homotopy equivalence.

Instead of having contractible fibres, the map of interest is required to be \emph{shrinkable}: it has a section that is also a fibrewise homotopy inverse. 
Shrinkability is thus the suitable continuous version of a family of spaces being contractible. 
This adjustment only affects the last part of the proof, and indeed most of the current proof is carefully setting up the `internal' (to the category of topological spaces) version of the first part of Quillen's proof from \cite{Quillen_73}.

Let $f\colon X \to Y$ be functor between well-pointed topological categories (Definition~\ref{well-pointed}). 
We introduce a topological category $Y_0 \sl f$ (Definition~\ref{slice_category}) that is analogous to a union of the lax fibres of $f$. 
There is a canonical functor $\rho\colon Y_0\sl f \to \disc(Y_0)$, where $\disc(Y_0)$ denotes the topological category with objects $Y_0$ and only identity arrows.

\begin{thma}
If $f\colon X\to Y$ is a functor such that $B\rho\colon B(Y_0\sl f) \to Y_0$ is shrinkable, then $Bf$ is a homotopy equivalence.
\end{thma}

Note that a shrinkable map (Definition~\ref{shrinkable}) is, amongst other things, an \emph{acyclic fibration}. The reverse implication is also true if the domain and codomain are cofibrant. 
I also give a variant of this theorem asking merely for $B\rho$ to be an acyclic fibration, under the same hypothesis on $X$ and $Y$:

\begin{thmaprime}
If $f\colon X\to Y$ is a functor such that $B\rho\colon B(Y_0\sl f) \to Y_0$ is an acyclic Serre fibration, then $Bf$ is a weak homotopy equivalence.
\end{thmaprime}

We can then apply Theorem A to essentially surjective and fully faithful functors (where essential surjectivity means: a certain map has local sections over a numerable cover). 
As one would hope, such functors give rise to homotopy equivalences on geometric realisation.
Throughout this paper we work with $\CGH$, the category of compactly generated Hausdorff spaces, to ensure geometric realisation commutes with fibred products \cite[Corollary 11.6]{May_72}.

Note that a different approach to Quillen's Theorem B is taken in \cite{Meyer} for more general homotopy colimits, given as 2-sided bar constructions.
The approach taken here is more elementary and, in places, allows for a stronger conclusion under suitable hypothesis.

Finally, the postscript details how the results so far apply to give a well-defined functor
\[
    \St^{pres}(\CGH,\mathcal{O}_{\mathrm{num}}) \to \CGH_2
\]
assigning to each topological stack of categories (with mild restrictions) a homotopy type; the codomain here is the 2-category with objects CGH spaces, arrows continuous maps, and 2-arrows homotopy classes of homotopies of maps. 
The objects in the domain of this functor might well be called ``topological piles'' (following Rezk \cite{Rezk_ICM}) or ``topological c-stack'' (following Drinfeld \cite{Drinfeld}), though the latter seems more euphonious.

\textbf{Acknowledgments:} The author thanks Danny Stevenson for suggesting this line of inquiry and patient discussions. 
Thanks also to Maxine Elena Calle for asking about the details of my Theorem A, which I had stated without proof on the $n$Lab, prompting the long-delayed release of this note, and Yuxun Sun, who pointed out an error in an early public release version of the paper.
An anonymous referee made a number of suggestions that helped tighten up the article where, in places, the text dated back to when I wrote this as a student.

\section{First constructions}

We first describe a number of categories which will appear in the proof of Theorem A. 
The constructions work for internal categories in any finitely complete ambient category (in particular for topological categories), but are also given here for categories in $\Set$ to clarify the nature of the objects and the arrows.

\begin{defn}
  For $Y$ any category, the category $TY$ is defined to be the strict pullback
  \[
    \xymatrix{
       TY  \ar[r]  \ar[d] & Y^{\mathbf{2}} \ar[d]^{\dom}
       \\
       \disc(Y_0) \ar[r]_-i & Y
    }
  \]
  where $\disc(Y_0)$ is the discrete category with objects $Y_0 := \Obj(Y)$, $i$ the canonical inclusion, $Y^{\mathbf{2}}$ is the arrow category of $Y$, and $\dom$ is the domain functor.
\end{defn}

Explicitly, objects of $TY$ are morphisms $g \colon a \to b$ in $Y$,  and morphisms $\xymatrix{g \ar[r]^h & g'}$ in  $TY$ are commuting triangles
\[
\xymatrix{
	a \ar[r]^g \ar@/_.5pc/[dr]_{g'}& b \ar[d]^h \\
	& b'
}
\]
in $Y$.
The category $TY$ can be imagined as the union of based path `spaces' over all basepoints.

\begin{defn}
Given a category $Y$, define the \emph{twisted arrow category} $\natural Y$ as follows. 
It has the same objects as $Y^\mathbf{2}$,
\[
    \Obj(\natural Y) = \{g \colon a \to b \in \Mor(Y)\} =: Y_1,
\]
but morphisms $\xymatrix{g \ar[r]^{(h,k)} & g'}$ the commuting squares
\[
    \xymatrix{
        a \ar[r]^{g} & b \ar[d]^{h} \\
        a' \ar[u]^{k} \ar[r]_{g'} & b'
    }
\]
in $Y$.
We compose in this category by pasting squares vertically.  
 \end{defn}

From the point of view of defining the twisted arrow category of $Y$ \emph{internally}, we can describe the underlying internal directed graph as follows:
\begin{itemize}
\item The object of objects is $Y_1$
\item The object of arrows is $Y_3 = Y_1\times_{Y_0}Y_1\times_{Y_0} Y_1$, the object of composable triples of morphisms of $Y$.
\item The source and target maps are the projection $\pr_2\colon Y_3 \to Y_1$ and the composition map $Y_3\to Y_1$, respectively.
\end{itemize}
It is then an easy exercise to define the internal composition map $Y_3\times_{Y_1} Y_3 \to Y_3$ (and the unit map $Y_1\to Y_3$).

Clearly there is an inclusion $TY \into \natural Y$, sending 
\[
\raisebox{20pt}{
    \xymatrix{
    	a \ar[r]^g \ar@/_.5pc/[dr]_{g'}& b \ar[d]^h \\
    	& b'
    }
}
\mapsto
\raisebox{20pt}{
    \xymatrix{
        a \ar[r]^{g} \ar@{=}[d] & b \ar[d]^{h} \\
        a \ar[r]_{g'} &  b'
    }}.
\]
We can similarly define a wide subcategory $T^oY\into \natural Y$ where the morphisms are only of the form $\xymatrix{g \ar[r]^{(\id,k)} & g'}$, hence diagrams of the form
\[
    \xymatrix{
        a \ar[r]^g & b  \\
        a' \ar[u]^k  \ar@/_.5pc/[ur]_{g'} & 
    }
\]
in $Y$. 
Notice that we have $T^oY =  T(Y^{\op})$.

There is a functor $\cod^\natural \colon \natural Y \to Y$, which sends $(x \to y) \mapsto y$ and
\[
\raisebox{20pt}{
    \xymatrix{
        a \ar[r]   &b \ar[d] \\
        a' \ar[r]\ar[u] &  b'
    }}
\mapsto
\raisebox{20pt}{
    \xymatrix{
        b \ar[d] \\
        b'
    }
}.
\]
This clearly restricts to a functor $\cod^T\colon TY \to Y$.

There is another functor $\dom^\natural \colon \natural Y \to Y^{\op}$, this time sending a morphism to its source, and a square to the left vertical map. 
This restricts to the functor $\dom^T\colon TY \to \disc(Y_0)$ sending $a\to b$ to $a$. 
There is a section $\sigma\colon \disc(Y_0) \to TY$ of $\dom^T$ that is also a left adjoint.
Similarly, there is a section $\tau\colon \disc(Y_0) \to T^oY$ of $\cod^{T^o}$ that is a right adjoint.

\begin{defn}
\label{slice_category}
Let $f\colon X \to Y$ be a functor. The category $Y_0 \sl f$ is defined as the strict pullback
\[
        \xymatrix{
        Y_0 \sl f \ar[r] \ar[d] & TY\ar[d]^{\cod^T} \\
        X \ar[r]_{f} & Y
    }
\]
\end{defn}

The objects of $Y_0 \sl f$ are pairs $(g\colon a \to f(b),b)$, for $g\in\Mor(Y)$ and $b\in X_0$, with morphisms a pair consisting of a commuting triangle 
\[
    \xymatrix{
    	a \ar[r]^g \ar@/_.5pc/[dr]_{g'}& f(b) \ar[d]^{f(h)} \\
    	& f(b')
    }
\]
in $Y$, and the arrow $h$ from $X$. 
This category acts like the union of the lax fibres of the functor $f$ at all basepoints.
Note that taking $f=\id_Y$ we get $Y_0\sl \id_Y = TY$.

\begin{defn}
Let $f\colon X \to Y$ be a functor. 
The category $S(f)$ is defined as the strict pullback
\[
    \xymatrix{
        S(f) \ar[r]^{\widehat{f}} \ar[d]_{Q_f} & \natural Y \ar[d]^{\cod^\natural} \\
        X \ar[r]_f & Y\,.
    }
\]
\end{defn}

Again taking $f=\id_Y$, we have $S(\id_Y) = \natural Y$.

We thus have spans of categories that fit into a commutative diagram
\begin{equation}
\label{map_of_spans}
    \raisebox{16pt}{
    \xymatrix{
        X \ar[d]_f &  \ar[l]_-{Q_f} S(f) \ar[d]_{\widehat{f}} \ar[r]^-{P_f}& Y^{\op}\ar@{=}[d] \\
        Y &\ar[l]^{\cod^\natural} \natural Y  \ar[r]_{\dom^\natural}& Y^{\op} 
    }
    }
\end{equation}
where $P_f$ is defined as the composite $\dom^\natural\circ \widehat{f}$.
We will show below that the functors $\cod^\natural$, $\dom^\natural$ and $Q_f$ are sent by geometric realisation to homotopy equivalences---and when $f$ satisfies the hypothesis in Theorem A, the same is true for $P_f$. 
Hence by the 2-out-of-3 property for homotopy equivalences, $\widehat{f}$ is then sent to a homotopy equivalence, and hence so is $f$.

\section{Classifying spaces}

We define the functor $B\colon \Cat \to \CGH$ to be the composite $|N(-)|$, where $N$ is the standard nerve construction and $|-|$ is geometric realisation. 
We will also use the same notation for the geometric realisation of the nerve of a topological category.

If a functor becomes a  homotopy equivalence when applying $B$, then we say the functor is a  homotopy equivalence. 
The following proposition is stated in \cite[\S1, Proposition 2]{Quillen_73}, for example, and holds for categories replaced by topological categories.

\begin{propn}\label{prop:natransf_homotopy}
If $\alpha\colon F \Rightarrow G\colon C \to D$ is a natural transformation, that is, a functor $\alpha\colon C \times \mathbf{2} \to D$, then $B\alpha$ is a homotopy from $BF$ to $BG$.
\end{propn}

As noted by Quillen, it follows that any functor with an adjoint is sent by geometric realisation to a homotopy equivalence, although the triangle identities are not needed for this conclusion.

\begin{example}
The maps $B\dom^T\colon BTY\to Y_0$ and $B\cod^{T^o}\colon BT^oY\to Y_0$ are homotopy equivalences.
\end{example}

Notice nothing specific has been said so far about topological categories---I haven't needed to because everything said so far works perfectly fine for categories internal to $\CGH$, in particular the functor $B$ extends to give a functor $\Cat(\CGH)\to \CGH$ (cf \cite[Corollary 11.6]{May_72}) that preserves finite limits, and which will denoted by the same letter. 
However, passing to topological categories doesn't go through completely without some alteration. 
Recall that a simplicial space is a simplicial object in $\CGH$. 
Our main source of simplicial spaces is by the nerves of topological categories, and it is good to know when we can expect to get a homotopy equivalence of geometric realisations of these nerves.

\begin{propn}[{\cite[Proposition A.1]{Segal_74}}]\label{prop:geom_realise_good_sspace}
If $f\colon M \to N$ is a map of simplicial topological spaces such that $f_j\colon M_j \to N_j$ is a  homotopy equivalence, and if the degeneracy maps of $M$ and $N$ are all closed cofibrations, then $|f|\colon |M| \to |N|$ is a  homotopy equivalence.
\end{propn}

Segal refers to a simplicial space $M$ satisfying the condition ``all the degeneracy maps of $M$ are closed cofibrations'' as a \emph{good} simplicial space. 
But this is at the level of simplicial spaces, and we are more interested in the intrinsic properties of topological categories. 
This boils down to talking about the identity-assigning map $e\colon C_0 \to C_1$ of a topological category $C$, from which all the degeneracy maps of $NC$ are formed.

We know that the nerve of a topological group, when considered as a one-object groupoid, is good if the inclusion of the identity element is a closed cofibration. 
When we pass to many elements, we need a relative version of cofibration, which we will define via a relative version of NDR-pairs.

\begin{defn}
Let $B$ be a space and $\CGH/B$ the category of spaces over $B$. 
A pair\footnote{We do not distinguish between the pair $(X,A)$ and the inclusion map $A\into X$.}
$(X,A)$ in $\CGH/B$ is an \emph{NDR-pair over $B$} if there are maps
\[
    \xymatrix{
        X \ar[rr]^u \ar[dr] & & I \times B \ar[dl]^{\pr_2} \\
        & B&
    }\qquad
    \xymatrix{
        X \times I \ar[rr]^h \ar[dr] & &X \ar[dl] \\
        & B&
    }
\]
such that:
\begin{enumerate}

    \item $A = u^{-1}(\{0\}\times B)$;

    \item For all $x\in X$, $h(x,0) = x$, and for all $(a,t)\in A\times I$, $h(a,t) = a$;

    \item For all $x\in u^{-1}([0,1)\times B)$, $h(x,1) \in A$.

\end{enumerate}
We say $(u,h)$ \emph{represent} $(X,A)$.
\end{defn}

As trivial example (needed for the subsequent lemma), consider an arbitrary (CGH) space $C\to B$ over $B$, and the pair $(C,\emptyset)$. 
The constant function $C\to [0,1]$, $c\mapsto 1$ and the constant homotopy on the identity map of $C$, $h(c,t) = c$, represent $(C,\emptyset)$ as an NDR-pair over $B$.
Note that it is not merely true that $A\into X$ is a closed cofibration, but that the inclusion map between fibres over any given $b\in B$ is also a closed cofibration. 
We need this formulation of NDR-pair in the slice category, so that the following product lemma holds.

\begin{lemma}\label{cofib_over}
If $(X,A)$ is an NDR-pair over $B$, $C \to B$ a space over $B$, then $(X\times_B C,A\times_B C)$ is an NDR-pair over $B$.
\end{lemma}

The proof of this lemma is a slight modification of the result in \cite[Chapter 6, \S4]{May} on products of NDR-pairs, taking the ``fibre product NDR-pair'' $(X,A)\times_B (C,\emptyset) = (X\times_B X, X\times_B \emptyset \cup A\times_B C) = (X\times_B C,A\times_B C)$.

\begin{propn}
\label{well-pointed--good}
Let $Y$ be a topological category. Then $NY$ is a good simplicial space if $(Y_1,Y_0)$ is an NDR-pair over $Y_0\times Y_0$ (using the diagonal and $(s,t)\colon Y_1 \to Y_0 \times Y_0$).
\end{propn}

\begin{proof}
Consider for each $0\leq i\leq p$ the maps 
\begin{multline*}
    Y_p \simeq Y_0 \times_{Y_0\times Y_0} (Y_i\times Y_{p-i}) \stackrel{s_i}{\To} Y_1 \times_{Y_0\times Y_0} (Y_i\times Y_{p-i}) \simeq Y_{p+1}\\
    \shoveleft{(y;y_0\xrightarrow{a_1} \cdots \xrightarrow{a_i} y;
        y\xrightarrow{b_1} z_1 \to \cdots \xrightarrow{b_{p-i}} z_{p-i})} \\
        \mapsto (y_0 \xrightarrow{a_1}\cdots \xrightarrow{a_i} y\xrightarrow{\id_y}
        y\xrightarrow{b_1} \cdots \xrightarrow{b_{p-i}} z_{p-i})
\end{multline*}
inserting an identity in a string of composable arrows.
\end{proof}

\begin{defn}
\label{well-pointed}
A topological category is called \emph{well-pointed} if the condition in Proposition \ref{well-pointed--good} holds.
\end{defn}

Examples of well-pointed topological categories include well-pointed topological monoids or groups. 
A well-pointed $\CGH$-\emph{enriched} category as defined in \cite{Vogt_73} is well-pointed in the sense above if considered as a category internal to $\CGH$. 
Note that $C$ is well-pointed iff $C^{\op}$ is well-pointed.

\section{A span of bisimplicial spaces}

Given a functor $f\colon X \to Y$ between topological categories let us define a bisimplicial space $D = D(f)$ by the following:
\[
    D(f)_{pq} = NY^{\op}_p \times_{Y_0} Y_1 \times_{Y_0} NX_q
\]
where the $(p,q)$-simplices look like
\[
    \big( y_p \to \ldots \to y_0 \stackrel{\eta}{\to} f(x_0) ; x_0  \to \ldots \to x_q \big),
\]
and we consider $p$ as indexing the vertical direction, and $q$ indexing the horizontal direction.
The face maps $d_i^h,d_i^v,\ i \geq 1$ are induced from $NX$ and $NY^{\op}$ with no effect on the $Y_1$ term. 
The face maps $d_0^h, d_0^v$ are
\begin{multline*}
d_0^h\big( y_p \to \ldots \to y_0 \stackrel{\eta}{\to} f(x_0) ; x_0  \stackrel{\nu}{\to} \ldots \to x_q \big)\\
    = \big( y_p \to \ldots \to y_0 \stackrel{f(\nu)\eta}{\longrightarrow} f(x_1) ; x_1  \to \ldots \to x_q \big) 
\end{multline*}
\begin{multline*}
    d_0^v\big( y_p \to \ldots \stackrel{\kappa}{\to} y_0 \stackrel{\eta}{\to} f(x_0) ; x_0  \to \ldots \to x_q \big) \\
    =
    \big( y_p \to \ldots \to y_1 \stackrel{\eta\kappa}{\longrightarrow} f(x_0) ; x_0  \to \ldots \to x_q \big).
\end{multline*}

\noindent
The degeneracy maps $s_i^h,s_i^v,\ i\geq 1$ are likewise induced from $NX$ and $NY^{\op}$.
The most important thing to note about the degeneracy maps is that they are all the identity on the $Y_1$ factor.
We get a span of bisimplicial spaces that on $(p,q)$ simplices looks like
\[
    NX_q \leftarrow D(f)_{pq} \to NY_p^{\op}
\]
where the codomains are constant in one direction.

For fixed $p$, the simplicial space $D_{p\bullet}$ is the fibred product of $N(Y_0\sl f$) and the constant simplicial space $NY_p^{\op}$. 
Thus the horizontal realisation $|D|_h$ of $D$ is a simplicial space with $p$-simplices
\begin{equation}\label{eq:ssSpace_span}
    (|D|_h)_p = NY^{\op}_p \times_{Y_0} B(Y_0\sl f),
\end{equation}
such that the degeneracy maps are given by the fibred product of the degeneracy maps for $NY^{\op}$ and the identity map for $B(Y_0\sl f)$. 
We can therefore apply Lemma~\ref{cofib_over} and so if $Y$ is well-pointed, then $NY^{\op}$ and hence $|D|_h$ are good.

There is a map of bisimplicial spaces $D \to NY^{\op}$ (where we think of the latter as being constant in the $q$-direction) which is simply the projection map. 
The following proposition is proved in \cite[\S1]{Quillen_73} and used in the special case that all spaces involved are discrete:

\begin{propn}\label{prop:bisimplicial_realisation}
If $T$ is a bisimplicial space, then there are natural isomorphisms
\[
    \big||T|_h\big| \simeq |dT| \simeq \big||T|_v\big|.
\]
\end{propn}

Here $d$ is the diagonal functor, which sends the bisimplicial space $\{T_{pq}\}$ to the simplicial space $\{T_{pp}\}$. 
We have constructed $D(f)$ such that $dD(f) = N(S(f))$, and as a result the diagonal functor applied to the span (\ref{eq:ssSpace_span}) gives (the nerve of) the top row of (\ref{map_of_spans}).
As a particular example of this, we have $dD(\id_Y) = N(\natural Y)$, and hence also get the bottom row of (\ref{map_of_spans}).

\begin{defn}[Dold \cite{Dold}]
\label{shrinkable}
A map $p\colon E \to B$ is \emph{shrinkable} if there is a section $s\colon  B \to E$ of $p$ such that $s\circ p$ is fibrewise homotopic to $\id_E$.
\end{defn}

The geometric realisation of a functor with a left or right adjoint section is shrinkable, as the (co)unit natural transformation geometrically realises to the required fibrewise homotopy.

When we apply the horizontal geometric realisation functor to the map $D \to NY^{\op}$ we get a map $\beta$ of simplicial spaces which at each level looks like
\[
    \beta_p\colon(|D|_h)_p = NY^{\op}_p \times_{Y_0} B(Y_0\sl f) \to NY^{\op}_p \times_{Y_0}Y_0 = NY^{\op}_p.
\]

Now if $B(Y_0 \sl f) \to Y_0$ is shrinkable, $\beta_p$ is a homotopy equivalence, as the pullback of a shrinkable map is shrinkable. 
Then, $\beta$ is a map of simplicial spaces which is a homotopy equivalence at each level. 
Given that $NY^{\op}$ and $NY^{\op} \times_{Y_0} B(Y_0\sl f)$ are good, we know that $|\beta|$ is a homotopy equivalence by Proposition~\ref{prop:geom_realise_good_sspace}.
But $|\beta|$ is secretly $BP_f$ (using Proposition~\ref{prop:bisimplicial_realisation}), the map we wanted to show was a homotopy equivalence.
Further, in the special case that we take $f=\id_Y$, then as $BTY \to Y_0$ is shrinkable (since $TY\to \disc(Y_0)$ has a left adjoint section), then the map $B\dom^\natural \colon B\natural Y \to BY^{\op}$ is also a homotopy equivalence.
And we also have the analogous result for $B\cod^\natural \colon B\natural Y \to BY$, since $\cod^{T^o}\colon T^oY\to \disc(Y_0)$ has a right adjoint section.

Finally, we need to show that $BQ_f\colon BS(f) \to BX$ is a homotopy equivalence.
This follows since the \emph{vertical} realisation $|D|_v$ is the simplicial space $BT^oY\times_{Y_0} NX$, which is good if $X$ is well-pointed, and the projection to $NX$ is at each level the shrinkable map
\[
    BT^oY\times_{Y_0} NX_q \to Y_0 \times_{Y_0} NX_q = NX_q
\]
that is the pullback of the shrinkable map $BT^oY\to Y_0$.
Then applyng Propositions~\ref{prop:bisimplicial_realisation} and \ref{prop:geom_realise_good_sspace} again we have that $BQ_f\colon BS(f) = |dS(f)| \simeq \big||D|_v\big| \to BX$ is a homotopy equivalence.

Thus we have shown:

\begin{thma}
If $f\colon X \to Y$ is a functor between well-pointed topological categories such that $B\rho\colon B(Y_0 \sl f) \to Y_0$ is shrinkable, $Bf$ is a homotopy equivalence.
\end{thma}

We have thus reduced the problem of showing $BX$ is homotopy equivalent to $BY$ to showing that the classifying space of a single topological category is homotopy equivalent to a given space. 

We can in fact do better than this:

\begin{thmaprime}
If $f\colon X\to Y$ is a functor between well-pointed topological categories such that $B\rho\colon B(Y_0 \sl f) \to Y_0$ is an acyclic Serre fibration, $Bf$ is a weak homotopy equivalence.
\end{thmaprime}

\begin{proof}
A map of \emph{proper} simplicial spaces which is a weak equivalence in each dimension geometrically realises to a weak homotopy equivalence (see \cite[A.4]{May_74}), and a good simplicial space is proper \cite{Lewis}.
Further, acyclic Serre fibrations are stable under pullback (as they are characterised by a right lifting property), so that the map $\beta_p$ is an acyclic Serre fibration when $B\rho$ is one.
Then $|\beta|=BP_f$ is weak homotopy equivalence, and using the 2-out-of-3 property of weak homotopy equivalences in the geometric realisation of the diagram (\ref{map_of_spans}), so is $B\widehat{f}$, and hence so is $Bf$, as desired.
\end{proof}

To remove the condition that $Y$ is well-pointed, we would need to use \emph{fat realisation}, which models the homotopy colimit of a simplicial topological space. 
It is known that the fat realisation of a levelwise weak homotopy equivalence is a weak homotopy equivalence \cite[Appendix A]{Segal_74}, so this step works.
This then leads to thinking about how to commute homotopy colimits past each other, specifically, the diagonal then fat realisation, and fat horizontal realisation then fat realisation, at the cost of the two constructions only being weakly homotopy equivalent.

Private discussion with J.~Scherer leads me to believe it should be possible to generalise Theorem A' and remove the hypothesis of well-pointedness of $Y$, but I leave this as an challenge for the motivated reader.

\section{Weak equivalences give homotopy equivalences}

Recall, firstly, that an open cover which admits a subordinate partition of unity is called \emph{numerable} \cite{Dold}. 
Define the singleton Grothendieck pretopology $\mathcal{O}_{\mathrm{num}}$ on $\CGH$ of `numerable covers' to have the covering maps of a (CGH) topological space $M$ to be those local homeomorphisms $U:= \coprod_\alpha U_\alpha \to M$ arising from a numerable open cover $\{U_\alpha\}$ of $M$.

Secondly, a functor $f\colon X \to Y$ is \emph{fully faithful} if the diagram
\begin{equation}\label{ff}
\raisebox{18pt}{
    \xymatrix{
        X_1 \ar[r]^{f_1} \ar[d]_{(s,t)} & Y_1 \ar[d]^{(s,t)} \\
        X_0 \times X_0 \ar[r]_{f_0} & Y_0 \times Y_0
    }
}
\end{equation}
is a pullback, and is \emph{essentially $\mathcal{O}_{\mathrm{num}}$-surjective} if $\rho_0$ in the diagram
\begin{equation}\label{locally-split_eso}
\raisebox{18pt}{    
    \xymatrix{
        X_0 \ar[d]_{f_0}  &X_0 \times_{Y_0} Y_1^{iso} \ar[l] \ar[d] \ar@/^1.8pc/[dr]^{\rho_0}& \\
        Y_0 & Y_1^{iso}  \ar[l]^s \ar[r]_t&Y_0 
    }
}
\end{equation}
admits local sections relative to some numerable cover $\pi_0\colon U \to Y_0$. 
That is, there is a map $(s_X,s_Y)\colon U\to X_0 \times_{Y_0} Y_1^{iso}$ such that $t(s_Y(u))=\pi_0(u)$ for all $u\in U$.
Note that here $Y_1^{iso}\subseteq Y_1$ is the subspace of invertible arrows. 
Let $U^{[2]}$ denote the topological groupoid $U\times_M U \ST U$ attached to a numerable cover $U \to M$, which comes equipped with a functor $\pi\colon U^{[2]} \to \disc(M)$.

\begin{propn}[Segal \cite{Segal_68}]\label{prop:Segal_shrinkable}
If $U \to M$ is a numerable cover of a space $M$, then $B\pi\colon BU^{[2]} \to M$ is shrinkable.
\end{propn}

We now come to the main application of the paper, generalising Segal's result to functors between topological categories. 

\begin{thm}\label{thm:Morita}
If $f\colon X \to Y$ is a fully faithful, essentially $\mathcal{O}_{\mathrm{num}}$-surjective functor between well-pointed topological categories, then $Bf$ is a homotopy equivalence.
\end{thm} 

\begin{proof}
Because (\ref{ff}) is a pullback, there is an \emph{isomorphism} of topological categories\footnote{Recall that given a space $S$, the topological category $\codisc(S)$ has as objects $S$ and a unique morphism between any ordered pair of elements of $S$.}
\[
    X \simeq \codisc(X_0) \times_{\codisc(Y_0)} Y.
\]
It immediately follows that 
\[
    Y_0 \sl f \simeq \codisc(X_0) \times_{\codisc(Y_0)} TY,
\]
and using the local sections $s$ of $\rho_0$ we can construct a functor $\sigma\colon U^{[2]} \to \codisc(X_0) \times_{\codisc(Y_0)} TY$, in the following way.

First note that there is a map 
\begin{align*}
    U\times_{Y_0} U &\to X_0 \times_{Y_0} Y_1^{iso} \times_{Y_0} X_0 \times_{Y_0} Y_1^{iso}\\
    (u,v)& \mapsto (s_X(u),f(s_X(u)) \xrightarrow{s_Y(u)} \pi(u);s_X(v),f(s_X(v)) \xrightarrow{s_Y(v)} \pi(v))
\end{align*}
where $\pi(u) = \pi(v)$.
If we define $s'(u,v) = s_Y(v)^{-1}\circ s_Y(u)$, then this gives a commuting triangle
\[
    \xymatrix{
        \pi(u) \ar[r]^{s_Y(u)^{-1}} \ar@/_.5pc/[dr]_{s_Y(v)^{-1}}& f(s_X(u)) \ar[d]^{s'(u,v)} \\
        & f(s_X(v))
    }
\]
which is a morphism in $TY$. 
Together with the pair $(s_X(u),s_X(v))\in \Mor(\codisc(X_0))$ this defines a map 
\[
    \sigma_1\colon U\times_{Y_0} U=\Mor(U^{[2]}) \to \Mor(\codisc(X_0) \times_{\codisc(Y_0)} TY), 
\]
which is the morphism component of the functor 
\[
    \sigma\colon U^{[2]} \to \codisc(X_0) \times_{\codisc(Y_0)} TY.
\] 
The object component is given by the map $\sigma_0\colon U\to X_0\times_{Y_0}Y_1$ of the form
\[
    u \mapsto (s_X(u),\pi(u)  \xrightarrow{s_Y(u)^{-1}}f(s_X(u))),
\]
and one can easily check using the definition of $s'$ that this is indeed an internal functor.
Moreover, the functor $\sigma$ satisfies the property that $\rho\circ\sigma$ is equal to $\pi\colon U^{[2]}\to \disc(Y_0)$.

Let $P$ be the strict pullback
\[
    \xymatrix{
        P  \ar[r] \ar[d]_{p_U} & \codisc(X_0) \times_{\codisc(Y_0)} TY \ar[d]^\rho \\
        U^{[2]} \ar[r]_{\pi} & \disc(Y_0)\; .
    }
\]
The section $\sigma$ induces a section $\tau$ of $p_U$, for which there is a natural transformation $\tau\circ p_U \Rightarrow \id_P$; the component at $(u;\pi(u)\xrightarrow{g} f(x);x)$ is given by a tuple (representing a morphism of $P$) consisting of the pair $(u,u)\in \Mor(U^{[2]})$, the pair $(s_X(u),x)$, and the commuting triangle
\[
    \xymatrix{
        \pi(u) \ar[r]^{s_Y(u)^{-1}} \ar@/_.5pc/[dr]_{g}& f(s_X(u)) \ar[d]^{gs_Y(u)} \\
        & f(x)
    } 
\]
Thus we see that $Bp_U$ is a homotopy equivalence by Proposition \ref{prop:natransf_homotopy}, as $B\tau\circ Bp_U = \id_{BU^{[2]}}$ and the natural transformation gives rise to a homotopy from $Bp_U\circ B\tau$ to $\id_{BP}$.

But, further, the natural transformation $\tau\circ p_U \Rightarrow \id_{P}\colon P\to P$ satisfies the condition that the corresponding functor $h\colon P\times \mathbf{2} \to P$ makes the diagram
\[
    \xymatrix{
        P \times \mathbf{2} \ar[rr] \ar[dr] && P \ar[dl]^{p_U} \\
        &U^{[2]}
    }
\]
commute, where the left diagonal map is the composition of the projection and $p_U$. 
This is because the $U^{[2]}$-component of the natural transformation is the identity arrow on the object $u$.
Thus $Bh$ is a fibrewise homotopy over $BU^{[2]}$, and hence $Bp_U$ is shrinkable.

When we pass to geometric realisations,
\begin{equation}\label{pullback_of_homotopy_fibres}
\raisebox{16pt}{
    \xymatrix{
        BP \ar[r] \ar[d]_{Bp_U} & B(\codisc(X_0) \times_{\codisc(Y_0)} TY) \ar[d]^{B\rho} \\
        BU^{[2]} \ar[r] & Y_0
    }
}
\end{equation}
it turns out that $B\rho$ is a retract of $Bp_U$, as follows.
Recall that $B\pi$ is shrinkable (with section $s'$, say) by Proposition~\ref{prop:Segal_shrinkable}, as we assumed $U\to Y_0$ is a numerable cover and then applying the pullback pasting lemma to the diagram
\[
    \xymatrix{
        Y_0 \times_{BU^{[2]}} BP \ar[d]_{\pr_1} \ar[r]^-{\pr_2} & BP \ar[r]^-r \ar[d]_{Bp_U} & B(\codisc(X_0) \times_{\codisc(Y_0)} TY) \ar[d]^{B\rho} \\
        Y_0 \ar[r]_{s'}& BU^{[2]} \ar[r]_{\pi} & Y_0 
    }
\]
we see that $Y_0 \times_{BU^{[2]}} BP \simeq B(\codisc(X_0) \times_{\codisc(Y_0)} TY)$ over $Y_0$. 
Thus $\pr_1$ can be identified with $B\rho$, and the composite of the top horizontal arrows with identity map.
The section of $B\rho$ is $r\circ B\tau\circ s'$, and it is then a short calculation to check that $B\rho$ is shrinkable.
As a result we can apply Theorem A to $f$ to conclude that $Bf$ is a homotopy equivalence.
\end{proof}

Because every open cover of a paracompact Hausdorff space is refined by a numerable one, an immediate corollary is that if $Y_0$ is paracompact\footnote{Note that $Y_0$ is always Hausdorff by the choice of $\CGH$ as our category of spaces.}, then local sections of $\rho_0$ over any open cover will suffice for the conclusion of Theorem~\ref{thm:Morita}.

\section{Postscript: homotopy types of topological stacks of categories}

The previous sections were written as a precursor to the content of my PhD thesis later published as \cite{Roberts_12}, and so did not benefit from the idea of stacks presentable by internal categories and Pronk's notion of bicategorical localisation. 
With that machinery, one can give the following corollary to Theorem~\ref{thm:Morita} (using judicious amounts of Global Choice). Assume now that all topological categories are well-pointed.

\begin{corollary}\label{cor:classifying_space_stack_cats}
The classifying space 2-functor $B\colon \Cat(\CGH)\to \CGH_2$ extends to give a classifying space 2-functor for presentable stacks of categories on the site $(\CGH,\mathcal{O}_{\mathrm{num}})$.
\end{corollary}

Here $\Cat(\CGH)$ denotes the (2,2)-category of well-pointed categories internal to $\CGH$, and $\mathcal{O}_{\mathrm{num}}$ is the pretopology on $\CGH$ given by numerable open covers. 
Also, $\CGH_2$ denotes the (2,1)-category of compactly generated Hausdorff spaces, continuous maps, and homotopy classes of homotopies as 2-arrows. 
Equivalences in $\CGH_2$ are precisely homotopy equivalences.

A presentable stack of categories on a given site $(S,J)$ is any stack that is the image of the stackification of a prestack associated to a category internal to $S$, up to equivalence. 
The 2-category of these is denoted $\St^{pres}(S,J)$.
There is an essentially surjective and locally fully faithful 2-functor $\yon\colon\Cat(\CGH) \to \St^{pres}(\CGH,\mathcal{O}_{\mathrm{num}})$ sending a topological category to the stack on $\CGH$ that it presents. 
Presentable stacks of categories in the algebro-geometric/homotopy theoretic setting have been explored in \cite{Rezk_ICM} and \cite{Drinfeld} (see also \cite{Roberts_MO}).

\begin{proof}[Proof (of Corollary \ref{cor:classifying_space_stack_cats})]
The 2-functor $\yon$ is a bicategorical localisation of $\Cat(\CGH)$ at the fully faithful, essentially $\mathcal{O}_{\mathrm{num}}$-surjective functors (this follows by combining \cite{Roberts_12} and \cite{PronkWarren}). 
The 2-functor $B$ sends such functors to homotopy equivalences, by Theorem~\ref{thm:Morita}, hence to equivalences in $\CGH_2$. 
Thus by the definition of bicategorical localisation there is a 2-functor
\[
    \St^{pres}(\CGH,\mathcal{O}_{\mathrm{num}}) \to \CGH_2
\]
whose composite with $\Cat(\CGH) \to \St^{pres}(\CGH,\mathcal{O}_{\mathrm{num}})$ is isomorphic to $B$.
\end{proof}

Thus every presentable stack of categories on $(\CGH,\mathcal{O}_{\mathrm{num}})$ has a well-defined homotopy type. 
This corollary should be compared with results of Ebert \cite{Ebert} and Noohi \cite{Noohi_12}. 
The latter, in particular constructs a (weak) homotopy type for topological stacks of \emph{groupoids} on the full site of topological spaces with \emph{all} open covers (and so neither that result nor Corollary~\ref{cor:classifying_space_stack_cats}  is a proper generalisation of the other).

However, Ebert's classifying space/homotopy type functor is defined for topological stacks presented by topological groupoids $X$ where \emph{every} space $NX_n$, $n\geq 0$, in the nerve $NX$ is paracompact Hausdorff.
As the construction in this paper works for (well-pointed) topological categories, and only requires that the space $X_0=NX_0$ of \emph{objects} is paracompact, we very nearly have a generalisation of Ebert's construction in the well-pointed context---at least up to homotopy.
This is because Ebert uses fat realisation, which under our assumption of well-pointedness agrees with ordinary geometric realisation up to weak homotopy equivalence.

Noohi's homotopy type of a ``hoparacompact'' topological stack \cite[\S8.2]{Noohi_12} is closer to the construction here, as this notion amounts to paracompactness of the object space $X_0$ of \emph{some} presenting groupoid $X$, plus an assumption on the properties of the presentation map $X_0 \to \mathcal{X}$ down to the topological stack presented by $X$.
Note that the topological space used by Noohi to present the homotopy type uses a Milnor-style classifying space construction, rather than the usual geometric realisation (fat or otherwise).

The usefulness of the present construction, in the author's view, lies in the fact that one is checking only very `local' information to know that a topological stack has a well-defined homotopy type: at most knowing properties of the presenting space $X_0$ (which becomes the object space of the internal category) and the unit map $X_0\to X_1$.
This is to be compared to having to know properties of the presentation map $X_0 \to \mathcal{X}$---the map itself, in the case of Noohi's homotopy type, or the simplicial space induced by it, in the case of Ebert's.

% \bibliographystyle{amsalpha} %
% \bibliography{refs_Roberts_NYJM}

\end{document}